\documentclass{amsart}
\usepackage{amssymb}
\usepackage{amscd}
\usepackage{multicol}
\usepackage{tikz}
\usepackage{graphics}
\usepackage{shuffle}
\usepackage[OT2,T1]{fontenc}
\usepackage[all]{xy}
\usepackage{extarrows}
\usepackage{comment}
\DeclareSymbolFont{cyrletters}{OT2}{wncyr}{m}{n}
\DeclareMathSymbol{\Sha}{\mathalpha}{cyrletters}{"58}

\title[Shuffle-type product formulae of desingularized values of MZFs]
{Shuffle-type product formulae of desingularized values of multiple zeta-functions\\
}
\author{Nao Komiyama}

\address{Graduate School of Mathematics, Nagoya University, 
Furo-cho, Chikusa-ku, Nagoya 464-8602 Japan }
\email{m15027u@math.nagoya-u.ac.jp}
\date{\today}

\newtheorem{thm}{Theorem}[section]
\newtheorem{lmm}[thm]{Lemma}
\newtheorem{crl}[thm]{Corollary}
\newtheorem{prp}[thm]{Proposition}  

\theoremstyle{remark}
        

\numberwithin{equation}{section}
\theoremstyle{definition}
\newtheorem{dfn}[thm]{Definition}
\newtheorem{rem}[thm]{Remark}
\newtheorem{notation}[thm]{Notation}     

\newtheorem{exa}[thm]{Examples}

\newcommand{\Bold}[1]{\mbox{\boldmath #1}}

\begin{document}
\bibliographystyle{amsalpha+}
\maketitle

\begin{abstract}      
It is known that there are infinitely many singularities of multiple zeta functions and the special values at non-positive integer points are indeterminate. In order to give a suitable rigorous meaning of the special values there, Furusho, Komori, Matsumoto and Tsumura introduced desingularized values by using their desingularization method to resolve all singularities. On the other hand, Ebrahimi-Fard, Manchon and Singer introduced renormalized values by the renormalization method \`{a} la Connes and Kreimer and they showed that the values fulfill the shuffle-type product formula. In this paper, we show the shuffle-type product formulae for desingularized values.
\end{abstract}
\setcounter{section}{-1}
\section{Introduction}

In 1776, Euler (\cite{Euler}) considered certain power series, the so-called double zeta values, and showed several relations among them. More than 200 years later, the {\it multiple zeta value} (MZV for short) which is more general series
$$\zeta(k_1, \dots, k_r) := \sum_{0<m_1<\cdots<m_r}\frac{1}{m_1^{k_1}\cdots m_r^{k_r}}$$
converging for $k_1,\dots,k_r\in\mathbb{N}$ and $k_r>1$, was discussed by Ecalle (\cite{Eca}) in 1981. In 1990s, these values also came to be focused by Hoffman (\cite{Hof}) and Zagier (\cite{Zag}). MZVs admit an iterated integral expressions, which enable us to regard them as a period of a certain motives (\cite{DG}, \cite{Go} and \cite{Te}) and calculate the Kontsevich invariant in knot theory (\cite{LM}). MZVs are also related to mathematical physics (\cite{BK1} and \cite{BK2}).

MZVs are regarded as special values at positive integer points of the several variables complex analytic function, the {\it multiple zeta-function} (MZF for short), which is defined by
\begin{equation*}
	\zeta(s_1, \dots, s_r) := \sum_{0<m_1<\cdots<m_r}\frac{1}{m_1^{s_1}\cdots m_r^{s_r}}.
\end{equation*}
It converges absolutely in the region
\begin{equation*}
	\{(s_1,\dots,s_r)\in\mathbb{C}^r\ |\ \frak{R}(s_{r-k+1}+\cdots+s_r)>k\ (1\leq k\leq r)\}.
\end{equation*}
 In the early 2000s, Zhao (\cite{Zhao}) and Akiyama, Egami and Tanigawa (\cite{AET}) independently showed that MZF can be meromorphically continued to $\mathbb{C}^r$. Especially, in \cite{AET}, the set of all singularities of the function $\zeta(s_1,\dots,s_r)$ is determined as
	\begin{align*}
		&s_r=1,\nonumber\\
		&s_{r-1}+s_r=2,1,0,-2,-4,\dots,\\
		&s_{r-k+1}+\cdots+s_r=k-n\quad (3\leq k\leq r,\ n\in\mathbb{N}_0).\nonumber
	\end{align*}
Because almost all of integer points with non-positive arguments are located in the above singularities, the special values of MZF there are indeterminate in all cases except for $\zeta(-k)$ at $k\in\mathbb{N}_0$, and $\zeta(-k_1,-k_2)$ at $k_1,k_2\in\mathbb{N}_0$ with $k_1+k_2$ odd. Actually, giving a nice definition of ``$\zeta(-k_1,\dots,-k_r)$'' for $k_1,\dots,k_r\in\mathbb{N}_0$ is one of our most fundamental problems.

In order to resolve all infinitely many singularities of MZF, the desingularization method was introduced  by Furusho, Komori, Matsumoto and Tsumura in \cite{FKMT1}. By applying this method to $\zeta(s_1,\dots,s_r)$, they constructed the {\it desingularized MZF} $\zeta_r^{\rm des}(s_1,\dots,s_r)$ which is entire on the whole space $\mathbb{C}^r$. The functions are represented by finite linear combinations of shifted MZFs (cf. Proposition \ref{prp:1.1}.). The {\it desingularized value}
\begin{equation*}
\zeta_r^{\rm des}(-k_1,\dots,-k_r)\in\mathbb{C}
\end{equation*}
is defined to be the special value of $\zeta_r^{\rm des}(s_1,\dots,s_r)$ at $(s_1,\dots,s_r)=(-k_1,\dots,-k_r)$ for $k_1,\dots,k_r\in\mathbb{N}_0$ (see Definition \ref{def:1.2.1}). In \cite{FKMT1}, its generating function given by
\begin{equation}\label{eqn:0.4}
	Z_{\scalebox{0.5}{\rm FKMT}}(t_1,\dots,t_r) := \sum_{k_1,\dots,k_r=0}^{\infty}\frac{(-t_1)^{k_1}\cdots(-t_r)^{k_r}}{k_1!\cdots k_r!}\zeta_r^{\rm des}(-k_1,\dots,-k_r)
\end{equation}
 in $\mathbb{C}[[t_1,\dots,t_r]]$ was calculated and desingularized values were explicitly described in terms of the Bernoulli numbers (see Proposition \ref{prp:1.1.1}.).

In contrast, Connes and Kreimer (\cite{CK}) started a Hopf algebraic approach to the renormalization procedure in the perturbative quantum field theory. A fundamental tool in their work is the {\it algebraic Birkhoff decomposition}. By applying this decomposition to a certain Hopf algebra related to MZVs, Guo and Zhang (\cite{GZ}) introduced the {\it renormalized values} which satisfy the harmonic-type product formulae. Later, Manchon and Paycha (\cite{MP}) and Ebrahimi-Fard, Manchon and Singer (\cite{EMS2}) introduced the different renormalized values which obey the harmonic-type product formulae by using different Hopf algebras. Ebrahimi-Fard, Manchon and Singer (\cite{EMS1}) also introduced another type of the renormalized values satisfying the shuffle-type product. We denote their values as
\begin{equation*}
\zeta_{\scalebox{0.5}{\rm EMS}}(-k_1,\dots,-k_r)\in\mathbb{C}
\end{equation*}
for $k_1,\dots,k_r\in \mathbb{N}_0$, and their generating function as
\begin{equation}\label{eqn:0.5}
	Z_{\scalebox{0.5}{\rm EMS}}(t_1,\dots,t_r) := \sum_{k_1,\dots,k_r=0}^{\infty}\frac{(-t_1)^{k_1}\cdots(-t_r)^{k_r}}{k_1!\cdots k_r!}\zeta_{\scalebox{0.5}{\rm EMS}}(-k_1,\dots,-k_r)
\end{equation}
 in $\mathbb{C}[[t_1,\dots,t_r]]$. In the paper \cite{Komi}, the author revealed the following relationship between generating functions (\ref{eqn:0.4}) and (\ref{eqn:0.5}):
\begin{equation}\label{eqn:0.8}
	Z_{\scalebox{0.5}{\rm EMS}}(t_1,\dots,t_r) = \prod_{i=1}^{r}\frac{1-e^{-t_i-\cdots-t_r}}{t_i+\cdots+t_r}\cdot Z_{\scalebox{0.5}{\rm FKMT}}(-t_1,\dots,-t_r).
\end{equation}
The following recurrence formulae were essential for the proof of the equation (\ref{eqn:0.8}):
\begin{equation}\label{eqn:0.11}
	Z_{\scalebox{0.5}{\rm FKMT}}(t_1,\dots,t_r) = Z_{\scalebox{0.5}{\rm FKMT}}(t_2,\dots,t_r)\cdot Z_{\scalebox{0.5}{\rm FKMT}}(t_1+\cdots+t_r).
\end{equation}
By these recurrence formulae, we will show our main theorem (Theorem \ref{prop:2.1}) that $\zeta_r^{\rm des}(-k_1,\dots,-k_r)$ fulfills the following shuffle-type product formula (\ref{thm:0.1}) which are shown to hold for $\zeta_{\scalebox{0.5}{\rm EMS}}(-k_1,\dots,-k_r)$ by \cite{EMS1}:\\

\noindent
\smallskip
{\bf Theorem \ref{prop:2.1}}
{\it For $k_1,\dots,k_p,l_1,\dots,l_q\in\mathbb{N}_0$, we have}
	\begin{align}\label{thm:0.1}
		&&\zeta_p^{\rm des}(-k_1,\dots,-k_p)\zeta_q^{\rm des}(-l_1,\dots,-l_q)\hspace{6.6cm} \\
		&&= \sum_{\substack{i_1 + j_1=l_1\\ \scalebox{0.5}{\rotatebox{90}{$\cdots$}}\\i_q + j_q=l_q}}\prod_{a=1}^q(-1)^{i_a}\binom{l_a}{i_a} \zeta_{p+q}^{\rm des}(-k_1,\dots,-k_{p-1},-k_p-i_1- \cdots -i_q,-j_1,\dots,-j_q). \nonumber
	\end{align}
The above recurrence formula (\ref{eqn:0.11}) also yields
\begin{equation}\label{eqn:0.9}
\zeta_r^{\rm des}(-k_1,\dots,-k_r)=\sum_{\substack{i+j=k_r \\
	i,j\geq0}}
	\binom{k_r}{i}\zeta_{r-1}^{\rm des}(-k_1,\dots,-k_{r-2},-k_{r-1}-i)\zeta_1^{\rm des}(-j)
\end{equation}
for $k_1,\dots,k_r\in\mathbb{N}_0$. We will extend the equation (\ref{eqn:0.9}) to the equation (\ref{eqn:4.1}) by replacing $-k_1,\dots,-k_{r-1}\in\mathbb{Z}_{\leq0}$ with $s_1,\dots,s_{r-1}\in\mathbb{C}$ in Proposition \ref{thm:3.1}.

The plan of our paper goes as follows. In \S1, we will review the algebraic Birkoff decomposition and the definition of the renormalized values in \cite{EMS1}. In \S2, we will recall the definition of the desingularized MZFs and the desingularized values introduced by Furusho, Komori, Matsumoto and Tsumura in \cite{FKMT1}. In \S3, we will prove the shuffle-type product formulae of desingularized values at non-positive integer points (Theorem \ref{prop:2.1}). In \S4, we will show the formula (\ref{eqn:4.2}) in Proposition \ref{thm:3.1}, which generalizes  the equation (\ref{thm:0.1}) in the case of $q=1$.

\section{Algebraic Birkhoff decomposition and renormalized values}
In this section, we assume $\mathcal{H}$ is a Hopf algebra over $\mathbb{Q}$, $\mathcal{A}:=\mathbb{Q}[[z]][z^{-1}]$ and $\mathcal{L}(\mathcal{H},\mathcal{A}):=\{f:\mathcal{H}\rightarrow\mathcal{A}\ |\ \mbox{$f$ is a $\mathbb{Q}$-linear map}\}$. For the maps $f,g\in\mathcal{L}(\mathcal{H},\mathcal{A})$, we define the convolution $f*g\in\mathcal{L}(\mathcal{H},\mathcal{A})$ by
\begin{equation*}
	f*g:=m\circ(f\otimes g)\circ\Delta,
\end{equation*}
where $m$ is the product of $\mathcal{A}$ and $\Delta$ is the coproduct of $\mathcal{H}$. Then, the subset
\begin{equation*}
	G(\mathcal{H},\mathcal{A}):=\{\ f\in\mathcal{L}(\mathcal{H},\mathcal{A})|\ f(1)=1\}
\end{equation*}
forms a group with the above convolution product $*$.
\begin{thm}[\cite{CK}, \cite{EMS1}: {\bf the algebraic Birkhoff decomposition}]\label{thm:1.1}
	\ \\For $f\in G(\mathcal{H},\mathcal{A})$, there are unique linear maps $f_+:\mathcal{H}\rightarrow\mathbb{Q}[[z]]$ and $f_-:\mathcal{H}\rightarrow\mathbb{Q}[z^{-1}]$ with $f_-(1)=1\in\mathbb{Q}$ such that
	\begin{equation*}
		f=f_-^{-1}*f_+,
	\end{equation*}
	where $f_-^{-1}$ is the inverse element of $f_-$ in $G(\mathcal{H},\mathcal{A})$. Moreover the maps $f_-$ and $f_+$ form algebra homomorphisms if the map $f$ is an algebra homomorphism.
\end{thm}

Let $\mathbb{Q}\langle d,y\rangle$ be the $\mathbb{Q}$-vector space generated by the word (including 1) of $d$ and $y$. We define the $\mathbb{Q}$-algebra $(\mathbb{Q}\langle d,y\rangle,\shuffle_0)$ by the new product $\shuffle_0:\mathbb{Q}\langle d,y\rangle^{\otimes2}\rightarrow\mathbb{Q}\langle d,y\rangle$ which is a $\mathbb{Q}$-linear map recursively defined by 
\begin{align*}
	1\shuffle_0 w & :=w\shuffle_01:=w, \\
	yu\shuffle_0v & :=u\shuffle_0yv:=y(u\shuffle_0 v), \\
	du\shuffle_0 dv & :=d(u\shuffle_0 dv)-u\shuffle_0 d^2v,
\end{align*}
for words $u$,$v$ and $w$ of $d$ and $y$. This algebra $(\mathbb{Q}\langle d,y\rangle,\shuffle_0)$ forms a non-commutative algebra. We consider the following set $\mathcal{S}$:
\begin{equation*}
	\mathcal{S}:=\langle d^k\{d(u\shuffle_0v)-du\shuffle_0v-u\shuffle_0dv\},\ wd\ |\ \mbox{$u,v,w$: words, $k\in\mathbb{N}_0$}\rangle_{(\mathbb{Q}\langle d,y\rangle,\shuffle_0)},
\end{equation*}
that is, to be the two-sided ideal of $(\mathbb{Q}\langle d,y\rangle,\shuffle_0)$ algebraically generated by the above elements. Then, the quotient
\begin{equation*}
	\mathcal{H}_0:=\mathbb{Q}\langle d,y\rangle/\mathcal{S},
\end{equation*}
forms a commutative and cocommutative Hopf algebra (its coproduct is not the deconcatenation coproduct. For detail, see \cite{EMS1},\cite{Komi}). We define the $\mathbb{Q}$-linear map $\phi:\mathcal{H}_0\rightarrow\mathcal{A}$ by $\phi(1):=1$ and for $k_1,\dots,k_n\in\mathbb{N}_0$,
\begin{equation*}
	\phi(d^{k_1}y\cdots d^{k_r}y)(z):=\partial^{k_1}_z(x\partial^{k_2}_z)\cdots (x\partial^{k_r}_z)(x(z)),
\end{equation*}
where $x:=x(z):=\frac{e^z}{1-e^z}\in\mathbb{Q}[[z]][z^{-1}]$ and $\partial_z$ is the derivative by $z$.
\begin{prp}[{\rm \cite[\S4.2]{EMS1}}]\label{prp:2.1}
	The $\mathbb{Q}$-linear map $\phi:\mathcal{H}_0\rightarrow\mathbb{Q}[[z]][z^{-1}]$ is well-defined and forms algebra homomorphism.
\end{prp}
By applying Theorem \ref{thm:1.1} to this map $\phi$, we obtain the algebra homomorphism $\phi_+:\mathcal{H}_0\rightarrow\mathbb{Q}[[z]]$.
\begin{dfn}[{\rm \cite[\S4.2]{EMS1}}]
	The renormalized values\footnote{If we follow the notations of \cite{EMS1}, it should be denoted by $\zeta_+(-k_r,\dots,-k_1)$.} $\zeta_{\scalebox{0.5}{\rm EMS}}(-k_1,\dots, -k_r)$ is defined by
	\begin{equation*}
		\zeta_{\scalebox{0.5}{\rm EMS}}(-k_1,\dots, -k_r):=\lim_{z\rightarrow0}\phi_+(d^{k_r}y\cdots d^{k_1}y)(z)
	\end{equation*}
	for $k_1,\dots,k_r\in\mathbb{N}_0$.
\end{dfn}
These renormalized values coincide with special values of the meromorphic continuation of MZFs at non-positive integer points which are non-singular, i.e.,
\begin{prp}[{\rm \cite[Theorem 4.3]{EMS1}}]
For $k\in\mathbb{N}_0$, we have
\begin{equation*}
	\zeta_{\scalebox{0.5}{\rm EMS}}(-k)=\zeta(-k),
\end{equation*}
and for $k_1,k_2\in\mathbb{N}_0$ with $k_1+k_2$ odd, we have
\begin{equation*}
	\zeta_{\scalebox{0.5}{\rm EMS}}(-k_1, -k_2)=\zeta(-k_1, -k_2).
\end{equation*}
\end{prp}
By Theorem \ref{thm:1.1} and Proposition \ref{prp:2.1}, we get the proposition below:
\begin{prp}[{\rm \cite[\S4.2]{EMS1}}: {\bf shuffle-type product formula}]\label{prop:2.3.1}
	\ \\For the elements $w$ and $w'$ of $\mathcal{H}_0$, we have
	$$\phi_+(w\shuffle_0w')=\phi_+(w)\phi_+(w').$$
\end{prp}
Here are examples in lower depth:
\begin{exa}
	For $a,b,c \in \mathbb{N}_0$, we have 
	\begin{align*}
		\zeta_{\scalebox{0.5}{\rm EMS}}(-a)\cdot\zeta_{\scalebox{0.5}{\rm EMS}}(-b)&=\sum_{k=0}^a(-1)^k\binom{a}{k}\zeta_{\scalebox{0.5}{\rm EMS}}(-b-k,-a+k),\\
		\zeta_{\scalebox{0.5}{\rm EMS}}(-a)\cdot\zeta_{\scalebox{0.5}{\rm EMS}}(-b,-c)&=\sum_{\substack{i_1+j_1=b \\
			i_2+j_2=c}}(-1)^{i_1+i_2}\binom{b}{i_1}\binom{c}{i_2}\zeta_{\scalebox{0.5}{\rm EMS}}(-a-i_1-i_2,-j_1,-j_2).
	\end{align*}
\end{exa}
In the paper \cite{Komi}, the author showed the explicit formula of $\zeta_{\scalebox{0.5}{\rm EMS}}(-k_1,\dots, -k_n)$:
\begin{prp}[\cite{Komi}]
For $r\in\mathbb{N}$, we have
	\begin{equation*}
		Z_{\scalebox{0.5}{\rm EMS}}(t_1,\dots,t_r)=\prod_{i=1}^r\frac{(t_i+\cdots+t_r)-(e^{t_i+\cdots+t_r}-1)}{(t_i+\cdots+t_r)(e^{t_i+\cdots+t_r}-1)},
	\end{equation*}
	where $Z_{\scalebox{0.5}{\rm EMS}}(t_1,\dots,t_r)$ is the generating function {\rm (\ref{eqn:0.5})} of $\zeta_{\scalebox{0.5}{\rm EMS}}(-k_1,\dots, -k_r)$.
\end{prp}

\section{Desingularization of multiple zeta-functions}
In this section, we review desingularized values introduced by Furusho, Komori, Matsumoto and Tsumura in \cite{FKMT1}. In \S1.1, we recall the definition of the desingularized MZF, and explain  some remarkable properties of this function. In \S1.2, we review desingularized values and their generating function.

\subsection{Desingularized MZFs}
In this subsection, we review the definition of desingularized MZF and the two properties, i.e., desingularized MZF can be analytically continued to $\mathbb{C}^r$ as an entire function (Proposition \ref{prp:1.2}) and can be represented by a finite ``linear'' combination of MZFs (Proposition \ref{prp:1.1}). We consider the generating function\footnote{It is denoted by $\tilde{\mathfrak{H}}_n\left((t_j);(1);c\right)$ in \cite{FKMT1}.} $\tilde{\mathfrak{H}}_r\left(t_1,\dots,t_r;c\right) \in \mathbb{C}[[t_1,\dots,t_r]]$ (cf. \cite[Definition 1.9]{FKMT1}):
\begin{align*}
	\tilde{\mathfrak{H}}_r\left(t_1,\dots,t_r;c\right)&:=\prod_{j=1}^r\left(\frac{1}{\exp{\left(\sum_{k=j}^r t_k\right)}-1}-\frac{c}{\exp{\left(c\sum_{k=j}^r t_k\right)}-1}\right)\\
	&=\prod_{j=1}^r\left(\sum_{m=1}^{\infty}(1-c^m)B_m\frac{\left(\sum_{k=j}^r t_k\right)^{m-1}}{m!}\right)
\end{align*}
for $c\in\mathbb{R}$. Here $B_m\ (m\geq0)$ is the Bernoulli number which is defined by
\begin{equation}\label{eqn:1.1.4}
\displaystyle\frac{x}{e^x-1}:=\sum_{m\geq0}\frac{B_m}{m!}x^m.
\end{equation}
We note that $B_0=1$, $B_1=-\frac{1}{2}$, $B_2=\frac{1}{6}$.
\begin{dfn}[{\rm \cite[Definition 3.1]{FKMT1}}]
	For non-integral complex numbers $s_1,\dots,s_r$, {\it desingularized MZF} $\zeta_r^{\rm des}(s_1,\dots,s_r)$ is defined by
	\begin{align}
		\label{eqn:1.1.2}&\zeta_r^{\rm des}(s_1,\dots,s_r) \\
		&:=\lim_{\substack{c\rightarrow1\\c\in\mathbb{R}\setminus\{1\}}}\frac{1}{(1-c)^r}\prod_{k=1}^r\frac{1}{(e^{2\pi is_k}-1)\Gamma(s_k)}\int_{\mathcal{C}^r}\tilde{\mathfrak{H}}_r\left(t_1,\dots,t_r;c\right)\prod_{k=1}^r t_k^{s_k-1}d t_k. \nonumber
	\end{align}
	Here $\mathcal{C}$ is the path consisting of the positive real axis (top side), a circle around the origin of radius $\varepsilon$ (sufficiently small), and the positive real axis (bottom side).
\end{dfn}
One of the remarkable properties of desingularized MZF is that it is an entire function, i.e., the equation (\ref{eqn:1.1.2}) is well-defined as an analytic function by the following proposition.
\begin{prp}[{\rm \cite[Theorem 3.4]{FKMT1}}]\label{prp:1.2}
	The equation $\zeta_r^{\rm des} (s_1,\dots,s_r)$ can be analytically continued to $\mathbb{C}^r$ as an entire function in $(s_1,\dots,s_r)\in \mathbb{C}^r$ by the following integral expression:
	\begin{align*}
		\label{eqn:1.1.2}\zeta_r^{\rm des}&(s_1,\dots,s_r) 
		=\prod_{k=1}^r\frac{1}{(e^{2\pi is_k}-1)\Gamma(s_k)}\\
		&\cdot\int_{\mathcal{C}^n}\prod_{j=1}^r\lim_{\substack{c\rightarrow1\\c\in\mathbb{R}\setminus\{1\}}}\frac{1}{1-c}\left(\frac{1}{\exp{\left(\sum_{k=j}^r t_k\right)}-1}-\frac{c}{\exp{\left(c\sum_{k=j}^r t_k\right)}-1}\right)\prod_{k=1}^r t_k^{s_k-1}d t_k.
	\end{align*}
\end{prp}
For indeterminates $u_j$ and $v_j\ (1\leq j\leq r)$, we set
\begin{equation}\label{eqn:1.1.3}
	\mathcal{G}_r(u_1,\dots,u_r; v_1,\dots,v_r):=\prod_{j=1}^r\left\{1-(u_jv_j+\cdots+u_r v_r)(v_j^{-1}-v_{j-1}^{-1})\right\}
\end{equation}
with the convention $v_0^{-1}:=0$, and we define the set of integers $\{a^r_{\Bold{\footnotesize$l$},\Bold{\footnotesize$m$}}\}$ by
\begin{equation}\label{eqn:1.1.4}
	\mathcal{G}_r(u_1,\dots,u_r; v_1,\dots,v_r)=\sum_{\substack{\mbox{\boldmath {\footnotesize$l$}}=(l_j)\in\mathbb{N}_0^r\\ \mbox{\boldmath {\footnotesize$m$}}=(m_j)\in\mathbb{Z}^r \\ |\mbox{\boldmath {\footnotesize$m$}}|=0}}a^r_{\mbox{\boldmath {\footnotesize$l$}},\mbox{\boldmath {\footnotesize$m$}}}\prod_{j=1}^ru_j^{l_j}v_j^{m_j}.
\end{equation}
Here, $|\mbox{\boldmath {\footnotesize$m$}}|:=m_1+\cdots+ m_r$.\\
Another remarkable properties of desingularized MZF is that the function is given by a finite ``linear'' combination of shifted MZFs, i.e.,
\begin{prp}[{\rm \cite[Theorem 3.8]{FKMT1}}]\label{prp:1.1}
	For $s_1,\dots,s_r \in \mathbb{C}$, we have the following equality between meromorphic functions of the complex variables $(s_1,\ldots,s_r)$:
	\begin{equation}\label{eqn:1.1.5}
		\zeta_r^{\rm des}(s_1,\dots,s_r)=\sum_{\substack{\mbox{\boldmath {\footnotesize$l$}}=(l_j)\in\mathbb{N}_0^r\\ \mbox{\boldmath {\footnotesize$m$}}=(m_j)\in\mathbb{Z}^r \\ |\mbox{\boldmath {\footnotesize$m$}}|=0}}a^r_{\mbox{\boldmath {\footnotesize$l$}},\mbox{\boldmath {\footnotesize$m$}}}\left(\prod_{j=1}^r(s_j)_{l_j}\right)\zeta(s_1+m_1,\dots,s_r+m_r).
	\end{equation}
	Here, $(s)_{k}$ is the {\it Pochhammer symbol}, that is, for $k\in\mathbb{N}$ and $s\in\mathbb{C}$ $(s)_{0}:=1$ and $(s)_k:=s(s+1)\cdots(s+k-1)$.
\end{prp}
\subsection{Desingularized values}
We review the definition of desingularized values and their explicit formula (Proposition \ref{prp:1.1.1}), and then we give a recurrence formula of desingularized values (Corollary \ref{crl:1.1.1}).

\begin{dfn}\label{def:1.2.1}
	For $k_1,\dots,k_r \in \mathbb{N}_0$, {\it desingularized value} $\zeta_r^{\rm des}(-k_1,\dots,-k_r)\in\mathbb{C}$ is defined to be the special value of desingularized MZF $\zeta_r^{\rm des}(s_1,\dots,s_r)$ at $(s_1,\dots,s_r)=(-k_1,\dots,-k_r)$.
\end{dfn}
The generating function $Z_{\scalebox{0.5}{\rm FKMT}}(t_1,\dots,t_r)$ of $\zeta_r^{\rm des}(-k_1,\dots,-k_r)$ in the equation (\ref{eqn:0.4}) is explicitly calculated as follows.
\begin{prp}[{\rm \cite[Theorem 3.7]{FKMT1}}]\label{prp:1.1.1}
	We have
	\begin{equation*}
		Z_{\scalebox{0.5}{\rm FKMT}}(t_1,\dots,t_r) = \prod_{i=1}^r\frac{(1-t_i-\cdots-t_r)e^{t_i+\cdots+t_r}-1}{(e^{t_i+\cdots+t_r}-1)^2}.
	\end{equation*}
	In terms of $\zeta_r^{\rm des}(-k_1,\dots,-k_r)$ for $k_1,\dots,k_r\in\mathbb{N}_0$, the above equation is reformulated to
	\begin{equation*}
		\zeta_r^{\rm des}(-k_1,\dots,-k_r)=(-1)^{k_1+\cdots+k_r}\sum_{\substack{\nu_{1i}+\cdots+\nu_{ii}=k_i\\1\leq i\leq r}}\prod_{i=1}^r\frac{k_i!}{\prod_{j=i}^r\nu_{ij}!}B_{\nu_{ii}+\cdots+\nu_{ir}+1}.
	\end{equation*}
\end{prp}
By the above proposition we have the following recurrence formula:
\begin{crl}\label{crl:1.1.1}
	\begin{equation}\label{eqn:1.2.1}
		Z_{\scalebox{0.5}{\rm FKMT}}(t_1,\dots,t_r) = Z_{\scalebox{0.5}{\rm FKMT}}(t_2,\dots,t_r)\cdot Z_{\scalebox{0.5}{\rm FKMT}}(t_1+\cdots+t_r) \quad(r \in \mathbb{N}).
	\end{equation}
	In terms of $\zeta_r^{\rm des}(-k_1,\dots,-k_r)$, the equation {\rm (\ref{eqn:1.2.1})} is reformulated to
	\begin{equation}\label{eqn:1.2.2}
		\zeta_r^{\rm des}(-k_1,\dots,-k_r) = \sum_{\substack{i_2 + j_2=k_2\\\scalebox{0.5}{\rotatebox{90}{$\cdots$}}\\i_r + j_r=k_r}}\prod_{a=2}^r\binom{k_a}{i_a}\zeta_{r-1}^{\rm des}(-i_2,\dots,-i_r)\zeta_1^{\rm des}(-k_1-j_2-\dots-j_r)
	\end{equation}
	for $k_1,\dots,k_r \in \mathbb{N}_0$.
\end{crl}
In the paper \cite{Komi}, the author showed that the desingularized values $\zeta_r^{\rm des}(-k_1,\dots,-k_r)$ and renormalized values $\zeta_{\scalebox{0.5}{\rm EMS}}(-k_1,\dots,-k_r)$ in \cite{EMS1} are equivalent (the equation (\ref{eqn:0.8})), i.e.
\begin{thm}
For $r\in\mathbb{N}$, we have
	\begin{equation*}
		Z_{\scalebox{0.5}{\rm EMS}}(t_1,\dots,t_r) = \prod_{i=1}^{r}\frac{1-e^{-t_i-\cdots-t_r}}{t_i+\cdots+t_r}\cdot Z_{\scalebox{0.5}{\rm FKMT}}(-t_1,\dots,-t_r).
	\end{equation*}
\end{thm}
The following is an example of our equivalence:
\begin{exa}\label{ex:3.1}
	For $k \in \mathbb{N}_0$, we have
	\begin{align*}
		&\zeta_{\scalebox{0.5}{\rm EMS}}(-k) = \displaystyle\sum_{i+j=k}\binom{k}{i}\frac{(-1)^j}{i+1}\zeta_{\scalebox{0.5}{\rm FKMT}}(-j),\\
		&\zeta_{\scalebox{0.5}{\rm FKMT}}(-k) = (-1)^{k}\displaystyle\sum_{i+j=k}\binom{k}{i}B_{i}\zeta_{\scalebox{0.5}{\rm EMS}}(-j).
	\end{align*}
\end{exa}

\section{The product formulae at non-positive integer points}
In this section, we prove the shuffle-type product formulae of desingularized values at non-positive integer points (Theorem \ref{prop:2.1}). 
\begin{lmm}\label{lmm:1.1}
	For $r\in\mathbb{N}$, we have
	\begin{equation}\label{eqn:1.1}
		Z_{\scalebox{0.5}{\rm FKMT}}(u_1)\cdots Z_{\scalebox{0.5}{\rm FKMT}}(u_r)=Z_{\scalebox{0.5}{\rm FKMT}}(u_1-u_2,u_2-u_3,\dots,u_{r-1}-u_r,u_r).
	\end{equation}
\end{lmm}
\begin{proof}
	Let $r \in \mathbb{N}$. Using the equation (\ref{eqn:1.2.1}) repeatedly, we get
	\begin{equation*}
		Z_{\scalebox{0.5}{\rm FKMT}}(t_1,\dots,t_r) = \prod_{i=1}^r Z_{\scalebox{0.5}{\rm FKMT}}(t_i+\cdots+t_r).
	\end{equation*}
	We replace $t_i + \cdots +t_r$ to $u_i$ for $i=1,\dots,r$ in this formula. Then, we obtain the equation (\ref{eqn:1.1}).
\end{proof}
Calculating simply, we obtain the following lemma.
\begin{lmm}\label{lmm:1.2}
For $r\in\mathbb{N}$, $a_1,\dots,a_r\in \mathbb{C}$ and $f:\mathbb{N}_0\rightarrow\mathbb{C}$, we have
\begin{align*}
	\sum_{k=0}^{\infty}\frac{(a_1+\cdots+a_r)^k}{k!}f(k)&=\sum_{k=0}^{\infty}\frac{f(k)}{k!}\sum_{i_1+\cdots+i_r=k}\frac{k!}{i_1!\cdots i_r!}a_1^{i_1}\cdots a_r^{i_r} \\
	&=\sum_{i_1,\dots,i_r=0}^{\infty}\frac{a_1^{i_1}\cdots a_r^{i_r}}{i_1!\cdots i_r!}f(i_1+\cdots+i_r).
\end{align*}
\end{lmm}
Using the above two lemmas, we have the following theorem.
\begin{thm}\label{prop:2.1}
For $p,q\in\mathbb{N}$ and $k_1,\dots,k_p,l_1,\dots,l_q\in\mathbb{N}_0$, we have
	\begin{align}\label{eqn:1.2}
		&&\zeta_p^{\rm des}(-k_1,\dots,-k_p)\zeta_q^{\rm des}(-l_1,\dots,-l_q)\hspace{6.6cm} \\
		&&= \sum_{\substack{i_1 + j_1=l_1\\ \scalebox{0.5}{\rotatebox{90}{$\cdots$}}\\i_q + j_q=l_q}}\prod_{a=1}^q(-1)^{i_a}\binom{l_a}{i_a} \zeta_{p+q}^{\rm des}(-k_1,\dots,-k_{p-1},-k_p-i_1- \cdots -i_q,-j_1,\dots,-j_q). \nonumber
	\end{align}
\end{thm}
\begin{proof}
	Using the equation (\ref{eqn:1.2.1}) repeatedly, we get
{\small 	
	\begin{equation*}
		Z_{\scalebox{0.5}{\rm FKMT}}(s_1,\dots,s_p)Z_{\scalebox{0.5}{\rm FKMT}}(t_1,\dots,t_q)
		= Z_{\scalebox{0.5}{\rm FKMT}}(s_1+\cdots+s_p) \cdots Z_{\scalebox{0.5}{\rm FKMT}}(s_p)Z_{\scalebox{0.5}{\rm FKMT}}(t_1+\cdots+t_q)\cdots Z_{\scalebox{0.5}{\rm FKMT}}(t_q).
	\end{equation*}
	By putting $u_i = \left\{\begin{array}{cc}
			s_i+\cdots+s_p & (1\leq i \leq p), \\
			t_{i-p}+\cdots+t_q & (p+1\leq i \leq p+q),
		\end{array}\right.$ and applying the equation (\ref{eqn:1.1}) to the above equation, we have
	\begin{align*}
		&Z_{\scalebox{0.5}{\rm FKMT}}(s_1,\dots,s_p)Z_{\scalebox{0.5}{\rm FKMT}}(t_1,\dots,t_q) \\
		=& Z_{\scalebox{0.5}{\rm FKMT}}(s_1,\dots,s_{p-1},s_p-t_1-\cdots-t_q,t_1,\dots,t_q) \\
		=& \sum_{k_1,\dots,k_p\geq0}\frac{(-s_1)^{k_1}\cdots(-s_{p-1})^{k_{p-1}}(-s_p+t_1+\cdots+t_q)^{k_p}}{k_1!\cdots k_{p-1}!k_p!} \\
		&\hspace{5em}\cdot\sum_{j_1,\dots,j_q\geq0}\frac{(-t_1)^{j_1}\cdots(-t_q)^{j_q}}{j_1!\cdots j_q!}\zeta_{p+q}^{\rm des}(-k_1,\dots,-k_p,-j_1,\dots,-j_q) \\
		=& \sum_{\substack{k_1,\dots,k_{p-1}\geq0\\j_1,\dots,j_q\geq0}}\frac{(-s_1)^{k_1}\cdots(-s_{p-1})^{k_{p-1}}}{k_1!\cdots k_{p-1}!}\frac{(-t_1)^{j_1}\cdots(-t_q)^{j_q}}{j_1!\cdots j_q!} \\
		&\hspace{5em}\cdot\sum_{k_p\geq0}\frac{(-s_p+t_1+\cdots+t_q)^{k_p}}{k_p!}\zeta_{p+q}^{\rm des}(-k_1,\dots,-k_p,-j_1,\dots,-j_q). \\
		\intertext{Using Lemma \ref{lmm:1.2}, we get}
		&Z_{\scalebox{0.5}{\rm FKMT}}(s_1,\dots,s_p)Z_{\scalebox{0.5}{\rm FKMT}}(t_1,\dots,t_q) \\
		=& \sum_{\substack{k_1,\dots,k_{p-1}\geq0\\j_1,\dots,j_q\geq0}}\frac{(-s_1)^{k_1}\cdots(-s_{p-1})^{k_{p-1}}}{k_1!\cdots k_{p-1}!}\frac{(-t_1)^{j_1}\cdots(-t_q)^{j_q}}{j_1!\cdots j_q!} \\
		&\hspace{3em}\cdot\sum_{k_p,i_1,\dots,i_q\geq0}\frac{(-s_p)^{k_p}t_1^{i_1}\cdots t_q^{i_q}}{k_p!i_1!\cdots i_q!}\zeta_{p+q}^{\rm des}(-k_1,\dots,-k_p-i_1-\cdots-i_q,-j_1,\dots,-j_q) \\
		=& \sum_{k_1,\dots,k_p\geq0}\frac{(-s_1)^{k_1}\cdots(-s_p)^{k_p}}{k_1!\cdots k_p!} \\
		&\hspace{1em}\cdot\sum_{\substack{i_1,\dots,i_q\geq0\\j_1,\dots,j_q\geq0}}\frac{t_1^{i_1}\cdots t_q^{i_q}}{i_1!\cdots i_q!}\frac{(-t_1)^{j_1}\cdots(-t_q)^{j_q}}{j_1!\cdots j_q!}\zeta_{p+q}^{\rm des}(-k_1,\dots,-k_p-i_1-\cdots-i_q,-j_1,\dots,-j_q) 
	\end{align*}
	\begin{align*}
		&Z_{\scalebox{0.5}{\rm FKMT}}(s_1,\dots,s_p)Z_{\scalebox{0.5}{\rm FKMT}}(t_1,\dots,t_q) \\
		=& \sum_{k_1,\dots,k_p\geq0}\frac{(-s_1)^{k_1}\cdots(-s_p)^{k_p}}{k_1!\cdots k_p!} \\
		&\hspace{0em}\cdot\sum_{\substack{i_1,\dots,i_q\geq0\\j_1,\dots,j_q\geq0}}\frac{(-t_1)^{i_1+j_1}\cdots(-t_q)^{i_q+j_q}}{i_1!\cdots i_q!j_1!\cdots j_q!}(-1)^{i_1+\cdots+i_q}\zeta_{p+q}^{\rm des}(-k_1,\dots,-k_p-i_1-\cdots-i_q,-j_1,\dots,-j_q) \\
		=& \sum_{\substack{k_1,\dots,k_p\geq0\\l_1,\dots,l_q\geq0}}\frac{(-s_1)^{k_1}\cdots(-s_p)^{k_p}}{k_1!\cdots k_p!}\frac{(-t_1)^{l_1}\cdots(-t_q)^{l_q}}{l_1!\cdots l_q!} \\
		&\hspace{0em}\cdot\sum_{\substack{i_1 + j_1=l_1\\ \scalebox{0.5}{\rotatebox{90}{$\cdots$}}\\i_q + j_q=l_q}}\prod_{a=1}^q\binom{l_a}{i_a}(-1)^{i_a} \zeta_{p+q}^{\rm des}(-k_1,\dots,-k_p-i_1-\cdots-i_q,-j_1,\dots,-j_q). \\
	\end{align*}}
	On the other hand, by the definition of $Z_{\scalebox{0.5}{\rm FKMT}}(t_1,\dots,t_q)$, we have
{\small 
	\begin{align*}
		Z_{\scalebox{0.5}{\rm FKMT}}&(s_1,\dots,s_p)Z_{\scalebox{0.5}{\rm FKMT}}(t_1,\dots,t_q) \\
		=& \left\{\sum_{k_1,\dots,k_p\geq0}\frac{(-s_1)^{k_1}\cdots(-s_p)^{k_p}}{k_1!\cdots k_p!}\zeta_p^{\rm des}(-k_1,\dots,-k_p)\right\} \\
		&\hspace{10em}\cdot\left\{\sum_{l_1,\dots,l_q\geq0}\frac{(-t_1)^{l_1}\cdots(-t_q)^{l_q}}{l_1!\cdots l_q!}\zeta_q^{\rm des}(-l_1,\dots,-l_q)\right\} \\
		=& \sum_{\substack{k_1,\dots,k_p\geq0\\l_1,\dots,l_q\geq0}}\frac{(-s_1)^{k_1}\cdots(-s_p)^{k_p}}{k_1!\cdots k_p!}\frac{(-t_1)^{l_1}\cdots(-t_q)^{l_q}}{l_1!\cdots l_q!}\zeta_p^{\rm des}(-k_1,\dots,-k_p)\zeta_q^{\rm des}(-l_1,\dots,-l_q). \\
	\end{align*}}
	Therefore, we obtain the equation (\ref{eqn:1.2}).
\end{proof}
Here are examples for $(p,q)=(1,1),\ (1,2)$.
\begin{exa}
For $a,b,c\in\mathbb{N}_0$, we have
	\begin{align*}
		\zeta_1^{\rm des}(-a)\zeta_1^{\rm des}(-b)&=\sum_{i_1+j_1=b}(-1)^{i_1}\binom{b}{i_1}\zeta_2^{\rm des}(-a-i_1,-j_1), \\
		\zeta_1^{\rm des}(-a)\zeta_2^{\rm des}(-b,-c)&=\sum_{\substack{i_1+j_1=b \\
			i_2+j_2=c}}(-1)^{i_1+i_2}\binom{b}{i_1}\binom{c}{i_2}\zeta_3^{\rm des}(-a-i_1-i_2,-j_1,-j_2).
	\end{align*}
\end{exa}

\begin{rem}
In order to prove Theorem \ref{prop:2.1}, we essentially used only the property (\ref{eqn:1.2.1}) of $Z_{\scalebox{0.5}{\rm FKMT}}(t_1,\dots,t_r)$, which also holds for  $Z_{\scalebox{0.5}{\rm EMS}}(t_1,\dots,t_r)$, so $\zeta_r^{\rm des}(-k_1,\dots,-k_r)$ satisfies the same shuffle-type product formula to $\zeta_{\scalebox{0.5}{\rm EMS}}(-k_1,\dots,-k_r)$ introduced in \cite{EMS1}.
\end{rem}

\section{More general product formulae}
In this section, we prove a generalization of the equation (\ref{eqn:0.9}) in Proposition \ref{thm:3.1} and {\it general} ``shuffle product'' between $\zeta_r^{\rm des}(s_1,\dots,s_{r-1})$ and $\zeta_1^{\rm des}(-l)$ in Proposition \ref{crl:3.2}. We assume $r\in\mathbb{N}_{\geq2}$ in this section. We start with the following lemma on the property of the Pochhammer symbol. 
\begin{lmm}\label{lmm:3.2}
	For $a,b\in\mathbb{C}$ and $n\in\mathbb{N}_0$, we have
	\begin{equation*}
		(a+b)_n=\sum_{i+j=n}\binom{n}{i}(a)_i(b)_j.
	\end{equation*}
\end{lmm}
\begin{proof}
	By considering the Taylor expansion of $(1-t)^{-a}$, we get
	\begin{equation*}
		(1-t)^{-a}=\sum_{n\geq0}\frac{(a)_n}{n!}t^n.
	\end{equation*}
	Because we have $(1-t)^{-a-b}=(1-t)^{-a}(1-t)^{-b}$, by comparing the coefficient of this equation, we obtain the claim.
\end{proof}
The above lemma is used in the proof of Proposition \ref{prp:3.1}. \\
Next, we prove a property of $\mathcal{G}_r((u_j);(v_j))$ defined by the equation (\ref{eqn:1.1.3}).

\begin{prp}\label{prp:3.1}
We have
	\begin{align}\label{eqn:3.1}
	\mathcal{G}_r&\left(u_1,\dots,u_r; v_1,\dots,v_{r-1},\frac{u_r+z}{u_r}v_{r-1}\right) \\
	&=(z+1)\mathcal{G}_{r-1}(u_1,\dots,u_{r-2},u_{r-1}+u_r+z; v_1,\dots,v_{r-1}). \nonumber
	\end{align}
\end{prp}
\begin{proof}
By the definition of $\mathcal{G}_r((u_j);(v_j))$, we have
\begin{align*}
	&\mathcal{G}_r\left(u_1,\dots,u_r; v_1,\dots,v_{r-1},\frac{u_r+z}{u_r}v_{r-1}\right) \\
	=&\prod_{j=1}^{r-1}\left\{1-\left(u_jv_j+\cdots+u_{r-1} v_{r-1}+u_r\frac{u_r+z}{u_r}v_{r-1}\right)(v_j^{-1}-v_{j-1}^{-1})\right\} \\
	&\cdot \left\{1-u_r\frac{u_r+z}{u_r}v_{r-1}\left(\left(\frac{u_r+z}{u_r}v_{r-1}\right)^{-1}-v_{r-1}^{-1}\right)\right\} \\
	=&\prod_{j=1}^{r-1}\left\{1-\left(u_jv_j+\cdots+u_{r-1} v_{r-1}+(u_r+z)v_{r-1}\right)(v_j^{-1}-v_{j-1}^{-1})\right\} \\
	&\cdot \left\{1-(u_r+z)v_{r-1}\left(\frac{u_r}{u_r+z}-1\right)v_{r-1}^{-1}\right\} 
\end{align*}
\begin{align*}
	=&\prod_{j=1}^{r-1}\left\{1-\left(u_jv_j+\cdots+u_{r-2} v_{r-2}+(u_{r-1}+u_r+z)v_{r-1}\right)(v_j^{-1}-v_{j-1}^{-1})\right\} \\
	&\cdot \left\{1-\left(u_r-(u_r+z)\right)\right\}\\
	=&(z+1)\mathcal{G}_{r-1}(u_1,\dots,u_{r-2},u_{r-1}+u_r+z; v_1,\dots,v_{r-1}). 
\end{align*}
\end{proof}

It is easy to prove the following lemma by comparing coefficients $a^r_{\Bold{\footnotesize$l$},\Bold{\footnotesize$m$}}$ of the equations (\ref{eqn:1.1.3}) and (\ref{eqn:1.1.4}).
\begin{lmm}\label{lmm:3.1}
	Let $\Bold{$l$}:=(l_j)\in\mathbb{N}_0^r$ and $\Bold{$m$}:=(m_j)\in\mathbb{Z}^r$. If $m_r\neq l_r-1, l_r$ or $m_r<0$, then we have
	\begin{equation}
		a^r_{\Bold{\footnotesize$l$},\Bold{\footnotesize$m$}}=0.
	\end{equation}
\end{lmm}

For our simplicities, we employ the following symbols:
\begin{notation}
Let $s_1,\dots,s_r$ and $z$ be indeterminates. For $r$-tuple symbol $\Bold{$s$}:=(s_1,\dots,s_r)$, the symbols $\Bold{$s$}'$ and $\Bold{$s$}^-$ are defined by
\begin{align*}
	\Bold{$s$}'&:=(s_1,\dots,s_{r-2},s_{r-1}+s_r), \\
	\Bold{$s$}^-&:=(s_1,\dots,s_{r-1}), \\
	|\Bold{$s$}|&:=s_1+\cdots+s_r,
\end{align*}
and we define $\Bold{$z$}:=(\underbrace{0,\dots,0}_{r-1},z)$.
\end{notation}

\begin{lmm}\label{lmm:3.3}
	For the functions $f:\mathbb{Z}^r\rightarrow\mathbb{C}$ and $g:\mathbb{N}_0^{r+1}\rightarrow\mathbb{C}$ with
	\begin{equation*}
		\#\{\Bold{$n$}\in\mathbb{Z}^r\ |\ f(\Bold{$n$})\neq0\}<\infty \ \mbox{and} \ \#\{\Bold{$a$}\in\mathbb{N}_0^{r+1}\ |\ g(\Bold{$a$})\neq0\}<\infty,
	\end{equation*}
	we have
	\begin{align}\label{eqn:3.7}
		\sum_{\substack{\Bold{\footnotesize$n$}=(n_j)\in\mathbb{Z}^r \\
			|\Bold{\footnotesize$n$}|=0}}f(\Bold{$n$})
			&=\sum_{\substack{\Bold{\footnotesize$m$}=(m_j)\in\mathbb{Z}^{r-1} \\
			|\Bold{\footnotesize$m$}|=0}}\sum_{\substack{p+q=m_{r-1} \\
			p,q\in\mathbb{Z}}}f(\Bold{$m$}^-,p,q), \\ 
		\label{eqn:3.8}
		\sum_{\Bold{\footnotesize$l$}=(l_j)\in\mathbb{N}_0^r}g(\Bold{$l$}',l_{r-1},l_r)
		&=\sum_{\Bold{\footnotesize$k$}=(k_j)\in\mathbb{N}_0^{r-1}}\sum_{\substack{p+q=k_{r-1} \\
			p,q\in\mathbb{N}_0}}g(\Bold{$k$},p,q).
	\end{align}
\end{lmm}
\begin{proof}
We only prove the equation (\ref{eqn:3.7}), because the proof of the equation (\ref{eqn:3.8}) can be done in the same way to that of the equation (\ref{eqn:3.7}). We have
	\begin{align*}
		\sum_{\substack{\Bold{\footnotesize$n$}=(n_j)\in\mathbb{Z}^r \\
			|\Bold{\footnotesize$n$}|=0}}f(\Bold{$n$})
			&=\sum_{n_1,\dots,n_{r-2},n_{r-1}\in\mathbb{Z}}f(n_1,\dots,n_{r-2},n_{r-1},-n_1-\cdots-n_{r-2}-n_{r-1}) \\
			&=\sum_{m_1,\dots,m_{r-2}\in\mathbb{Z}}\sum_{n_{r-1}\in\mathbb{Z}}f(m_1,\dots,m_{r-2},n_{r-1},-m_1-\cdots-m_{r-2}-n_{r-1}). \\
			\intertext{When we put $m_{r-1}:=-m_1-\cdots-m_{r-2}$, then $m_{r-1}$ can run over all integers. So we get}
			&=\sum_{\substack{\Bold{\footnotesize$m$}=(m_j)\in\mathbb{Z}^{r-1} \\
			|\Bold{\footnotesize$m$}|=0}}\sum_{n_{r-1}\in\mathbb{Z}}f(m_1,\dots,m_{r-2},n_{r-1},m_{r-1}-n_{r-1}). \\
			\intertext{When we put $p:=n_{r-1}$ and $q:=m_{r-1}-n_{r-1}$, then $p$ and $q$ run over all integers with $p+q=m_{r-1}$. So we obtain}
			&=\sum_{\substack{\Bold{\footnotesize$m$}=(m_j)\in\mathbb{Z}^{r-1} \\
			|\Bold{\footnotesize$m$}|=0}}\sum_{\substack{p+q=m_{r-1} \\
			p,q\in\mathbb{Z}}}f(\Bold{$m$}^-,p,q).
	\end{align*}
\end{proof}

Using Lemma \ref{lmm:3.1} and Lemma \ref{lmm:3.3}, we get the following corollary.
\begin{crl}\label{crl:3.1}
	For $\Bold{$l$}:=(l_1,\dots,l_r)\in\mathbb{N}_0^r$ and $\Bold{$m$}:=(m_1,\dots,m_{r-1})\in\mathbb{Z}^{r-1}$ with $|\Bold{$m$}|=0$, we have
	\begin{align}\label{eqn:3.2}
		a^r_{\Bold{\footnotesize$l$},(\Bold{\footnotesize$m$}^-,m_{r-1}-l_r,l_r)}+a^r_{\Bold{\footnotesize$l$},(\Bold{\footnotesize$m$}^-,m_{r-1}-l_r+1,l_r-1)}&=\binom{l_{r-1}+l_r}{l_{r-1}}a^{r-1}_{\Bold{\footnotesize$l$}',\Bold{\footnotesize$m$}} \\
		&=-a^r_{\left(\Bold{\footnotesize$l$}^-,l_r+1\right),(\Bold{\footnotesize$m$}^-,m_{r-1}-l_r,l_r)}. \nonumber
	\end{align}
\end{crl}
\begin{proof}
	Let $r\in\mathbb{N}$. By the equation (\ref{eqn:1.1.4}) (the definition of the coefficient $a^r_{\Bold{\footnotesize$l$},\Bold{\footnotesize$m$}}$ of the function $\mathcal{G}_r$), we have
	{\small
	\begin{align}\label{eqn:3.3}
		\mathcal{G}_r\left(u_1,\dots,u_r; v_1,\dots,v_{r-1},\frac{u_r+z}{u_r}v_{r-1}\right)
		=\sum_{\substack{\Bold{\footnotesize$l$}=(l_j)\in\mathbb{N}_0^r\\
			\Bold{\footnotesize$n$}=(n_j)\in\mathbb{Z}^r \\
			|\Bold{\footnotesize$n$}|=0}}
		a^r_{\Bold{\footnotesize$l$},\Bold{\footnotesize$n$}}\left(\prod_{j=1}^ru_j^{l_j}\right)\left(\prod_{j=1}^{r-1}v_j^{n_j}\right)\left(\frac{u_r+z}{u_r}v_{r-1}\right)^{n_r}.
	\end{align}}
	By using the equation (\ref{eqn:3.7}) of Lemma \ref{lmm:3.3}, we have
	\begin{align*}
		&\mathcal{G}_r\left(u_1,\dots,u_r; v_1,\dots,v_{r-1},\frac{u_r+z}{u_r}v_{r-1}\right) \\
		&=\sum_{\substack{\Bold{\footnotesize$l$}=(l_j)\in\mathbb{N}_0^r\\
			\Bold{\footnotesize$m$}=(m_j)\in\mathbb{Z}^{r-1} \\
			|\Bold{\footnotesize$m$}|=0}}
			\sum_{\substack{p+q=m_{r-1}\\
				p,q\in\mathbb{Z}}}
		a^r_{\Bold{\footnotesize$l$},(\Bold{\footnotesize$m$}^-,p,q)}\left(\prod_{j=1}^ru_j^{l_j}\right)\left(\prod_{j=1}^{r-2}v_j^{m_j}\right)v_{r-1}^p\left(\frac{u_r+z}{u_r}v_{r-1}\right)^q \\
		&=\sum_{\substack{\Bold{\footnotesize$l$}=(l_j)\in\mathbb{N}_0^r\\
			\Bold{\footnotesize$m$}=(m_j)\in\mathbb{Z}^{r-1} \\
			|\Bold{\footnotesize$m$}|=0}}
			\sum_{\substack{p+q=m_{r-1}\\
				p,q\in\mathbb{Z}}}
		a^r_{\Bold{\footnotesize$l$},(\Bold{\footnotesize$m$}^-,p,q)}\left(\prod_{j=1}^ru_j^{l_j}\right)\left(\frac{u_r+z}{u_r}\right)^q\left(\prod_{j=1}^{r-1}v_j^{m_j}\right).
	\end{align*}
	By Lemma \ref{lmm:3.1}, we get $a^r_{\Bold{\footnotesize$l$},(\Bold{\footnotesize$m$}^-,p,q)}=0$ for $q\neq l_r-1,l_r$. So we have
	\begin{align*}
		&\mathcal{G}_r\left(u_1,\dots,u_r; v_1,\dots,v_{r-1},\frac{u_r+z}{u_r}v_{r-1}\right) \\
		&=\sum_{\substack{\Bold{\footnotesize$l$}=(l_j)\in\mathbb{N}_0^r\\
			\Bold{\footnotesize$m$}=(m_j)\in\mathbb{Z}^{r-1} \\
			|\Bold{\footnotesize$m$}|=0}}
		\left\{a^r_{\Bold{\footnotesize$l$},(\Bold{\footnotesize$m$}^-,m_{r-1}-l_r+1,l_r-1)}\left(\prod_{j=1}^{r-1}u_j^{l_j}\right)u_r(u_r+z)^{l_r-1}\left(\prod_{j=1}^{r-1}v_j^{m_j}\right)\right. \nonumber\\
		&\hspace{3.5cm}\left.+ a^r_{\Bold{\footnotesize$l$},(\Bold{\footnotesize$m$}^-,m_{r-1}-l_r,l_r)}\left(\prod_{j=1}^{r-1}u_j^{l_j}\right)(u_r+z)^{l_r}\left(\prod_{j=1}^{r-1}v_j^{m_j}\right)\right\} \\
		&=\sum_{\substack{\Bold{\footnotesize$l$}=(l_j)\in\mathbb{N}_0^r\\
			\Bold{\footnotesize$m$}=(m_j)\in\mathbb{Z}^{r-1} \\
			|\Bold{\footnotesize$m$}|=0}}
		\left\{a^r_{\Bold{\footnotesize$l$},(\Bold{\footnotesize$m$}^-,m_{r-1}-l_r+1,l_r-1)}\left(\prod_{j=1}^{r-1}u_j^{l_j}\right)(-z)(u_r+z)^{l_r-1}\left(\prod_{j=1}^{r-1}v_j^{m_j}\right)\right. \nonumber\\
		&\hspace{3.2cm}+ a^r_{\Bold{\footnotesize$l$},(\Bold{\footnotesize$m$}^-,m_{r-1}-l_r+1,l_r-1)}\left(\prod_{j=1}^{r-1}u_j^{l_j}\right)(u_r+z)^{l_r}\left(\prod_{j=1}^{r-1}v_j^{m_j}\right) \nonumber\\
		&\hspace{3.2cm}\left.+ a^r_{\Bold{\footnotesize$l$},(\Bold{\footnotesize$m$}^-,m_{r-1}-l_r,l_r)}\left(\prod_{j=1}^{r-1}u_j^{l_j}\right)(u_r+z)^{l_r}\left(\prod_{j=1}^{r-1}v_j^{m_j}\right)\right\}.
	\end{align*}
		By Lemma \ref{lmm:3.1}, we get $a^r_{\Bold{\footnotesize$l$},(\Bold{\footnotesize$m$}^-,m_{r-1}+1,-1)}=0$ (i.e. the case of $l_r=0$). By replacing $l_r-1$ with $l_r$, we have
	\begin{align*}
		&\mathcal{G}_r\left(u_1,\dots,u_r; v_1,\dots,v_{r-1},\frac{u_r+z}{u_r}v_{r-1}\right) \\
		&=-z\sum_{\substack{\Bold{\footnotesize$l$}=(l_j)\in\mathbb{N}_0^r\\
			\Bold{\footnotesize$m$}=(m_j)\in\mathbb{Z}^{r-1} \\
			|\Bold{\footnotesize$m$}|=0}}
		 \left\{\vphantom{\cdot\left(\prod_{j=1}^{r-1}u_j^{l_j}\right)\left(\prod_{j=1}^{r-1}v_j^{m_j}\right)(u_r+z)^{l_r}}
		 a^r_{(\Bold{\footnotesize$l$}^-,l_r+1),(\Bold{\footnotesize$m$}^-,m_{r-1}-l_r,l_r)}\left(\prod_{j=1}^{r-1}u_j^{l_j}\right)(u_r+z)^{l_r}\left(\prod_{j=1}^{r-1}v_j^{m_j}\right)\right\} \\
		&\hspace{2.2cm}+\sum_{\substack{\Bold{\footnotesize$l$}=(l_j)\in\mathbb{N}_0^r\\
			\Bold{\footnotesize$m$}=(m_j)\in\mathbb{Z}^{r-1} \\
			|\Bold{\footnotesize$m$}|=0}}
		\left\{\vphantom{\cdot\left(\prod_{j=1}^{r-1}u_j^{l_j}\right)\left(\prod_{j=1}^{r-1}v_j^{m_j}\right)(u_r+z)^{l_r}}
		\left(a^r_{\Bold{\footnotesize$l$},(\Bold{\footnotesize$m$}^-,m_{r-1}-l_r,l_r)}+ a^r_{\Bold{\footnotesize$l$},(\Bold{\footnotesize$m$}^-,m_{r-1}-l_r+1,l_r-1)}\right)\right.\\
		&\hspace{6cm}\left.\cdot\left(\prod_{j=1}^{r-1}u_j^{l_j}\right)(u_r+z)^{l_r}\left(\prod_{j=1}^{r-1}v_j^{m_j}\right)\right\}.
	\end{align*}
On the other hand, we have
	\begin{align*}
		&(z+1)\mathcal{G}_{r-1}(u_1,\dots,u_{r-2},u_{r-1}+u_r+z; v_1,\dots,v_{r-1}) \\
		&=(z+1)\sum_{\substack{\Bold{\footnotesize$k$}=(k_j)\in\mathbb{N}_0^{r-1} \\
			\Bold{\footnotesize$m$}=(m_j)\in\mathbb{Z}^{r-1} \\
			|\Bold{\footnotesize$m$}|=0}}
		a^{r-1}_{\Bold{\footnotesize$k$},\Bold{\footnotesize$m$}}\left(\prod_{j=1}^{r-2}u_j^{k_j}\right)(u_{r-1}+u_r+z)^{k_{r-1}}\left(\prod_{j=1}^{r-1}v_j^{m_j}\right) \hspace{1.8cm} \\
		&=(z+1)\sum_{\substack{\Bold{\footnotesize$k$}=(k_j)\in\mathbb{N}_0^{r-1}\\
			\Bold{\footnotesize$m$}=(m_j)\in\mathbb{Z}^{r-1} \\
			|\Bold{\footnotesize$m$}|=0}}
		a^{r-1}_{\Bold{\footnotesize$k$},\Bold{\footnotesize$m$}}\left(\prod_{j=1}^{r-2}u_j^{k_j}\right)
			\sum_{\substack{p+q=k_{r-1} \\
			p,q\in\mathbb{N}_0}} \binom{k_{r-1}}{p}u_{r-1}^p(u_r+z)^q\left(\prod_{j=1}^{r-1}v_j^{m_j}\right).
	\end{align*}
	By using the equation (\ref{eqn:3.8}) of Lemma \ref{lmm:3.3}, we have
	{\small 
	\begin{align}\label{eqn:3.4}
		&(z+1)\mathcal{G}_{r-1}(u_1,\dots,u_{r-2},u_{r-1}+u_r+z; v_1,\dots,v_{r-1}) \\
		&=(z+1)\sum_{\substack{\Bold{\footnotesize$l$}=(l_j)\in\mathbb{N}_0^r\\
			\Bold{\footnotesize$m$}=(n_j)\in\mathbb{Z}^{r-1} \\
			|\Bold{\footnotesize$m$}|=0}} \binom{l_{r-1}+l_r}{l_{r-1}}
		a^{r-1}_{\Bold{\footnotesize$l$}',\Bold{\footnotesize$m$}}\left(\prod_{j=1}^{r-1}u_j^{l_j}\right)(u_r+z)^{l_r}\left(\prod_{j=1}^{r-1}v_j^{m_j}\right). \hspace{1.cm} \nonumber
	\end{align}}
By comparing the coefficients of (\ref{eqn:3.3}) and (\ref{eqn:3.4}), we obtain (\ref{eqn:3.2}).
	\end{proof}
\noindent
By tracing the proof of Corollary \ref{crl:3.1} inversely, we get the following proposition.
\begin{prp}\label{prp:3.1}
For $s_1,\dots,s_r,z\in\mathbb{C}$,
\begin{align}\label{eqn:3.5}
	\sum_{\substack{\Bold{\footnotesize$l$}=(l_j)\in\mathbb{N}_0^r\\
		\Bold{\footnotesize$n$}=(n_j)\in\mathbb{Z}^r \\
		|\Bold{\footnotesize$n$}|=0}}
	a^r_{\Bold{\footnotesize$l$},\Bold{\footnotesize$n$}}\left(\prod_{j=1}^r(s_j)_{l_j}\right)&\frac{\Gamma(s_r+n_r+z)\Gamma(-z)}{\Gamma(s_r+n_r)}\zeta_{r-1}(\Bold{$s$}'+\Bold{$n$}'+\Bold{$z$}') \\
	&=(1+z)\frac{\Gamma(s_r+z)\Gamma(-z)}{\Gamma(s_r)}\zeta_{r-1}^{\rm des}(\Bold{$s$}'+\Bold{$z$}'), \nonumber
\end{align}
holds except for singularities.
\end{prp}
\begin{proof}
	Let $s_1,\dots,s_r,z\in\mathbb{C}$. Using Corollary \ref{crl:3.1}, we have
	{\small
	\begin{align}\label{eqn:3.6}
		&(z+1)\sum_{\Bold{\footnotesize$l$}=(l_j)\in\mathbb{N}_0^r} \binom{l_{r-1}+l_r}{l_{r-1}}
		a^{r-1}_{\Bold{\footnotesize$l$}',\Bold{\footnotesize$m$}}\left(\prod_{j=1}^{r-1}(s_j)_{l_j}\right)(s_r+z)_{l_r} \\
		&=-z\sum_{\Bold{\footnotesize$l$}=(l_j)\in\mathbb{N}_0^r}
		 \left\{\vphantom{\cdot\left(\prod_{j=1}^{r-1}u_j^{l_j}\right)\left(\prod_{j=1}^{r-1}v_j^{m_j}\right)(u_r+z)^{l_r}}
		 a^r_{(\Bold{\footnotesize$l$}^-,l_r+1),(\Bold{\footnotesize$m$}^-,m_{r-1}-l_r,l_r)}\left(\prod_{j=1}^{r-1}(s_j)_{l_j}\right) (s_r+z)_{l_r}\right\} \nonumber\\
		&\quad+\sum_{\Bold{\footnotesize$l$}=(l_j)\in\mathbb{N}_0^r}
		\left\{\vphantom{\cdot\left(\prod_{j=1}^{r-1}u_j^{l_j}\right)\left(\prod_{j=1}^{r-1}v_j^{m_j}\right)(u_r+z)^{l_r}}
		\left(a^r_{\Bold{\footnotesize$l$},(\Bold{\footnotesize$m$}^-,m_{r-1}-l_r,l_r)}+ a^r_{\Bold{\footnotesize$l$},(\Bold{\footnotesize$m$}^-,m_{r-1}-l_r+1,l_r-1)}\right)\left(\prod_{j=1}^{r-1}(s_j)_{l_j}\right) (s_r+z)_{l_r}\right\}. \nonumber
	\end{align}}
\normalsize{By multiplying the function $\zeta_{r-1}(\Bold{$s$}'+\Bold{$m$}+\Bold{$z$}')$ and taking summation over $\Bold{$m$}\in\mathbb{Z}^{r-1}$ with $|\Bold{$m$}|=0$, we have}
	{\footnotesize
	\begin{align*}
		&\sum_{\substack{\Bold{\footnotesize$m$}=(m_j)\in\mathbb{Z}^{r-1} \\
			|\Bold{\footnotesize$m$}|=0}}
		\mbox{\normalsize{(R.H.S. of (\ref{eqn:3.6}))}}\cdot\zeta_{r-1}(\Bold{$s$}'+\Bold{$m$}+\Bold{$z$}') \\
		&=-z\sum_{\substack{\Bold{\footnotesize$l$}=(l_j)\in\mathbb{N}_0^r\\
			\Bold{\footnotesize$m$}=(m_j)\in\mathbb{Z}^{r-1} \\
			|\Bold{\footnotesize$m$}|=0}}
		 \left\{\vphantom{\cdot\left(\prod_{j=1}^{r-1}u_j^{l_j}\right)\left(\prod_{j=1}^{r-1}v_j^{m_j}\right)(u_r+z)^{l_r}}
		 a^r_{(\Bold{\footnotesize$l$}^-,l_r+1),(\Bold{\footnotesize$m$}^-,m_{r-1}-l_r,l_r)}\left(\prod_{j=1}^{r-1}(s_j)_{l_j}\right) (s_r+z)_{l_r}\zeta_{r-1}(\Bold{$s$}'+\Bold{$m$}+\Bold{$z$}')\right\}\\
			&\hspace{1cm}+\sum_{\substack{\Bold{\footnotesize$l$}=(l_j)\in\mathbb{N}_0^r\\
			\Bold{\footnotesize$m$}=(m_j)\in\mathbb{Z}^{r-1} \\
			|\Bold{\footnotesize$m$}|=0}}
		\left\{\vphantom{\cdot\left(\prod_{j=1}^{r-1}u_j^{l_j}\right)\left(\prod_{j=1}^{r-1}v_j^{m_j}\right)(u_r+z)^{l_r}}
		\left(a^r_{\Bold{\footnotesize$l$},(\Bold{\footnotesize$m$}^-,m_{r-1}-l_r,l_r)}+ a^r_{\Bold{\footnotesize$l$},(\Bold{\footnotesize$m$}^-,m_{r-1}-l_r+1,l_r-1)}\right)\right. \\
		&\hspace{5cm}\left.\cdot\left(\prod_{j=1}^{r-1}(s_j)_{l_j}\right) (s_r+z)_{l_r}\zeta_{r-1}(\Bold{$s$}'+\Bold{$m$}+\Bold{$z$}')\right\} \\
		&=\sum_{\substack{\Bold{\footnotesize$l$}=(l_j)\in\mathbb{N}_0^r\\
			\Bold{\footnotesize$m$}=(m_j)\in\mathbb{Z}^{r-1} \\
			|\Bold{\footnotesize$m$}|=0}}
		\left\{a^r_{\Bold{\footnotesize$l$},(\Bold{\footnotesize$m$}^-,m_{r-1}-l_r+1,l_r-1)}\left(\prod_{j=1}^{r-1}(s_j)_{l_j}\right) (-z)(s_r+z)_{l_r-1}\zeta_{r-1}(\Bold{$s$}'+\Bold{$m$}+\Bold{$z$}')\right. \\
		&\hspace{2.5cm}+ a^r_{\Bold{\footnotesize$l$},(\Bold{\footnotesize$m$}^-,m_{r-1}-l_r+1,l_r-1)}\left(\prod_{j=1}^{r-1}(s_j)_{l_j}\right) (s_r+z)_{l_r}\zeta_{r-1}(\Bold{$s$}'+\Bold{$m$}+\Bold{$z$}') \\
		&\hspace{3cm}\left.+ a^r_{\Bold{\footnotesize$l$},(\Bold{\footnotesize$m$}^-,m_{r-1}-l_r,l_r)}\left(\prod_{j=1}^{r-1}(s_j)_{l_j}\right) (s_r+z)_{l_r}\zeta_{r-1}(\Bold{$s$}'+\Bold{$m$}+\Bold{$z$}')\right\}  
		\end{align*}
		\begin{align*}
		&=\sum_{\substack{\Bold{\footnotesize$l$}=(l_j)\in\mathbb{N}_0^r\\
			\Bold{\footnotesize$m$}=(m_j)\in\mathbb{Z}^{r-1} \\
			|\Bold{\footnotesize$m$}|=0}}
		\left\{a^r_{\Bold{\footnotesize$l$},(\Bold{\footnotesize$m$}^-,m_{r-1}-l_r+1,l_r-1)}\left(\prod_{j=1}^{r-1}(s_j)_{l_j}\right) (s_r+l_r-1)(s_r+z)_{l_r-1}\zeta_{r-1}(\Bold{$s$}'+\Bold{$m$}+\Bold{$z$}')\right. \nonumber \\
		&\hspace{3cm}\left.+ a^r_{\Bold{\footnotesize$l$},(\Bold{\footnotesize$m$}^-,m_{r-1}-l_r,l_r)}\left(\prod_{j=1}^{r-1}(s_j)_{l_j}\right) (s_r+z)_{l_r}\zeta_{r-1}(\Bold{$s$}'+\Bold{$m$}+\Bold{$z$}')\right\}.  \nonumber \\
	\end{align*}}
	By Lemma \ref{lmm:3.1}, we get $a^r_{\Bold{\footnotesize$l$},(\Bold{\footnotesize$m$}^-,p,q)}=0$ for $q\neq l_r-1,l_r$. So we have
	\begin{align*}
		&\sum_{\substack{\Bold{\footnotesize$m$}=(m_j)\in\mathbb{Z}^{r-1} \\
			|\Bold{\footnotesize$m$}|=0}}
		\mbox{\normalsize{(R.H.S. of (\ref{eqn:3.6}))}}\cdot\zeta_{r-1}(\Bold{$s$}'+\Bold{$m$}+\Bold{$z$}') \nonumber\\
		&=\sum_{\substack{\Bold{\footnotesize$l$}=(l_j)\in\mathbb{N}_0^r\\
			\Bold{\footnotesize$m$}=(m_j)\in\mathbb{Z}^{r-1} \\
			|\Bold{\footnotesize$m$}|=0}}
			\sum_{\substack{p+q=m_{r-1}\\
				p,q\in\mathbb{Z}}}
		a^r_{\Bold{\footnotesize$l$},(\Bold{\footnotesize$m$}^-,p,q)}\left(\prod_{j=1}^r(s_j)_{l_j}\right) \frac{(s_r+z)_q}{(s_r)_q}\zeta_{r-1}(\Bold{$s$}'+\Bold{$m$}+\Bold{$z$}').
	\end{align*}
	By using the equation (\ref{eqn:3.7}) of Lemma \ref{lmm:3.3}, we have
	\begin{align}\label{eqn:4.10}
		&\sum_{\substack{\Bold{\footnotesize$m$}=(m_j)\in\mathbb{Z}^{r-1} \\
			|\Bold{\footnotesize$m$}|=0}}
		\mbox{\normalsize{(R.H.S. of (\ref{eqn:3.6}))}}\cdot\zeta_{r-1}(\Bold{$s$}'+\Bold{$m$}+\Bold{$z$}') \\
		&=\sum_{\substack{\Bold{\footnotesize$l$}=(l_j)\in\mathbb{N}_0^r\\
			\Bold{\footnotesize$n$}=(n_j)\in\mathbb{Z}^r \\
			|\Bold{\footnotesize$n$}|=0}}
		a^r_{\Bold{\footnotesize$l$},\Bold{\footnotesize$n$}}\left(\prod_{j=1}^r(s_j)_{l_j}\right)\frac{(s_r+z)_{n_r}}{(s_r)_{n_r}}\zeta_{r-1}(\Bold{$s$}'+\Bold{$n$}'+\Bold{$z$}'). \nonumber
	\end{align}
	We have $\Gamma(s+n)=(s)_n\Gamma(s)$ for $s\in\mathbb{C}$ and $n\in\mathbb{N}_0$, by the relation $\Gamma(s+1)=s\Gamma(s)$. By multiplying the equation (\ref{eqn:4.10}) with ${\Gamma(s_r+z)\Gamma(-z)}/{\Gamma(s_r)}$, we obtain
	\begin{align}\label{eqn:4.11}
		&\frac{\Gamma(s_r+z)\Gamma(-z)}{\Gamma(s_r)}\cdot (\mbox{L.H.S. of (\ref{eqn:4.10})}) \\
		&=\sum_{\substack{\Bold{\footnotesize$l$}=(l_j)\in\mathbb{N}_0^r\\
			\Bold{\footnotesize$n$}=(n_j)\in\mathbb{Z}^r \\
			|\Bold{\footnotesize$n$}|=0}}
		a^r_{\Bold{\footnotesize$l$},\Bold{\footnotesize$n$}}\left(\prod_{j=1}^r(s_j)_{l_j}\right)\frac{\Gamma(s_r+n_r+z)\Gamma(-z)}{\Gamma(s_r+n_r)}\zeta_{r-1}(\Bold{$s$}'+\Bold{$n$}'+\Bold{$z$}'). \nonumber
	\end{align}
	\normalsize{On the other hand, we have}
	\begin{align*}
		&\sum_{\substack{\Bold{\footnotesize$m$}=(m_j)\in\mathbb{Z}^{r-1} \\
			|\Bold{\footnotesize$m$}|=0}}
		\mbox{(L.H.S. of (\ref{eqn:3.6}))}\cdot\zeta_{r-1}(\Bold{$s$}'+\Bold{$m$}+\Bold{$z$}') \\
		&=(z+1)\sum_{\substack{\Bold{\footnotesize$l$}=(l_j)\in\mathbb{N}_0^r\\
			\Bold{\footnotesize$m$}=(m_j)\in\mathbb{Z}^{r-1} \\
			|\Bold{\footnotesize$m$}|=0}} \binom{l_{r-1}+l_r}{l_{r-1}}
		a^{r-1}_{\Bold{\footnotesize$l$}',\Bold{\footnotesize$m$}}\left(\prod_{j=1}^{r-1}(s_j)_{l_j}\right)(s_r+z)_{l_r}\zeta_{r-1}(\Bold{$s$}'+\Bold{$m$}+\Bold{$z$}').
		\intertext{By using the equation (\ref{eqn:3.8}) of Lemma \ref{lmm:3.3}, we have}
		&=(z+1)\sum_{\substack{\Bold{\footnotesize$k$}=(k_j)\in\mathbb{N}_0^{r-1}\\
			\Bold{\footnotesize$m$}=(m_j)\in\mathbb{Z}^{r-1} \\
			|\Bold{\footnotesize$m$}|=0}}
			\sum_{\substack{p+q=k_{r-1} \\
			p,q\in\mathbb{N}_0}} \binom{k_{r-1}}{p}
		a^{r-1}_{\Bold{\footnotesize$k$},\Bold{\footnotesize$m$}} \\
		&\hspace{4.5cm}\cdot\left(\prod_{j=1}^{r-2}(s_j)_{k_j}\right)(s_{r-1})_p(s_r+z)_q\zeta_{r-1}(\Bold{$s$}'+\Bold{$m$}+\Bold{$z$}').
	\end{align*}
		Using Lemma \ref{lmm:3.2}, we have
	{\small 
	\begin{align}\label{eqn:4.12}
		&\sum_{\substack{\Bold{\footnotesize$m$}=(m_j)\in\mathbb{Z}^{r-1} \\
			|\Bold{\footnotesize$m$}|=0}}
		\mbox{(L.H.S. of (\ref{eqn:3.6}))}\cdot\zeta_{r-1}(\Bold{$s$}'+\Bold{$m$}+\Bold{$z$}') \\
		&=(z+1)\sum_{\substack{\Bold{\footnotesize$k$}=(k_j)\in\mathbb{N}_0^{r-1} \\
			\Bold{\footnotesize$m$}=(m_j)\in\mathbb{Z}^{r-1} \\
			|\Bold{\footnotesize$m$}|=0}}
		a^{r-1}_{\Bold{\footnotesize$k$},\Bold{\footnotesize$m$}}\left(\prod_{j=1}^{r-2}(s_j)_{k_j}\right)(s_{r-1}+s_r+z)_{k_r-1}\zeta_{r-1}(\Bold{$s$}'+\Bold{$m$}+\Bold{$z$}'). \nonumber
	\end{align}}
		By multiplying the equation (\ref{eqn:4.12}) with ${\Gamma(s_r+z)\Gamma(-z)}/{\Gamma(s_r)}$ and by the equation (\ref{eqn:1.1.5}) of the desingularized function $\zeta_r^{\rm des}(\Bold{$s$})$, we obtain
	\begin{align}\label{eqn:4.13}
		\frac{\Gamma(s_r+z)\Gamma(-z)}{\Gamma(s_r)}\cdot(\mbox{R.H.S. of (\ref{eqn:4.12})})=(1+z)\frac{\Gamma(s_r+z)\Gamma(-z)}{\Gamma(s_r)}\zeta_{r-1}^{\rm des}(\Bold{$s$}'+\Bold{$z$}').
	\end{align}
So we obtain the equation (\ref{eqn:3.5}), by combining the equations (\ref{eqn:4.11}) and (\ref{eqn:4.13}) because we have $(\ref{eqn:4.10})=(\ref{eqn:4.12})$.
\end{proof}

\begin{prp}\label{thm:3.1}
For $s_1,\dots,s_{r-1}\in\mathbb{C}$ and $k\in\mathbb{N}_0$, we have 
\begin{equation}\label{eqn:4.2}
\zeta_r^{\rm des}(s_1,\dots,s_{r-1},-k)=\sum_{i+j=k}\binom{k}{i}\zeta_{r-1}^{\rm des}(s_1,\dots,s_{r-2},s_{r-1}-i)\zeta_1^{\rm des}(-j).
\end{equation}
\end{prp}
\begin{proof}
Let $\Bold{$s$}:=(s_1,\dots,s_r)\in\mathbb{C}^r$. By Mellin-Barnes integral formula, we obtain the following formula (\cite[the equation (3.7)]{Matsumoto});
\begin{equation*}
	\zeta_r(\Bold{$s$})=\frac{1}{2\pi i}\int_{(c)}\frac{\Gamma(s_r+z)\Gamma(-z)}{\Gamma(s_r)}\zeta_{r-1}(\Bold{$s$}'+\Bold{$z$}')\zeta(-z)dz,
\end{equation*}
for $\Re{(s_j)}>1\ (1\leq j\leq r)$, $-\Re{(s_r)}<c<0$ and the path of
integration is the vertical line $\Re{(z)}=c$. By this formula and the definition of $\zeta_r^{\rm des}(\Bold{$s$})$, we have
\begin{align*}
	\zeta_r^{\rm des}(\Bold{$s$})
	=&\sum_{\substack{\mbox{\boldmath {\footnotesize$l$}}=(l_j)\in\mathbb{N}_0^r\\ \mbox{\boldmath {\footnotesize$n$}}=(n_j)\in\mathbb{Z}^r \\ |\mbox{\boldmath {\footnotesize$n$}}|=0}}a^r_{\mbox{\boldmath {\footnotesize$l$}},\mbox{\boldmath {\footnotesize$n$}}}\left(\prod_{j=1}^r(s_j)_{l_j}\right)\zeta(\Bold{$s$}+\Bold{$n$}). \\
	=&\frac{1}{2\pi i}\int_{(c)}\sum_{\substack{\mbox{\boldmath {\footnotesize$l$}}=(l_j)\in\mathbb{N}_0^r\\ \mbox{\boldmath {\footnotesize$n$}}=(n_j)\in\mathbb{Z}^r \\ |\mbox{\boldmath {\footnotesize$n$}}|=0}}a^r_{\mbox{\boldmath {\footnotesize$l$}},\mbox{\boldmath {\footnotesize$n$}}}\left(\prod_{j=1}^r(s_j)_{l_j}\right)\frac{\Gamma(s_r+n_r+z)\Gamma(-z)}{\Gamma(s_r+n_r)} \\
	&\hspace{5cm}\cdot\zeta_{r-1}(\Bold{$s$}'+\Bold{$n$}'+\Bold{$z$}')\zeta(-z)dz. \\
	\intertext{Using Proposition \ref{prp:3.1}, we get}
	=&\frac{1}{2\pi i}\int_{(c)}(1+z)\frac{\Gamma(s_r+z)\Gamma(-z)}{\Gamma(s_r)}\zeta_{r-1}^{\rm des}(\Bold{$s$}'+\Bold{$z$}')\zeta(-z)dz. \\
	\intertext{By Proposition \ref{prp:1.1}, we have the formula $\zeta_1^{\rm des}(s)=(1-s)\zeta(s)$, so we obtain}
	=&\frac{1}{2\pi i}\int_{(c)}\frac{\Gamma(s_r+z)\Gamma(-z)}{\Gamma(s_r)}\zeta_{r-1}^{\rm des}(\Bold{$s$}'+\Bold{$z$}')\zeta_1^{\rm des}(-z)dz. \\
	\intertext{For $M\in\mathbb{N}$ and sufficiently small $\varepsilon>0$, we set $\mathcal{D}:=\{z\in\mathbb{C}\ |\ c<\Re{(z)}<M-\varepsilon\}$. For $z\in\mathcal{D}$, we have $\Re{(s_r+z)}>0$ by $-\Re{(s_r)}<c<0$. So singularities of the above integrand, which lie on $\mathcal{D}$, are only $z=0,1,2,\dots,M-1$. By using the residue theorem, we get}
	=&-\sum_{j=0}^{M-1}{\rm Res}\left[\frac{\Gamma(s_r+z)\Gamma(-z)}{\Gamma(s_r)}\zeta_{r-1}^{\rm des}(\Bold{$s$}'+\Bold{$z$}')\zeta_1^{\rm des}(-z),z=j\right] \\
	&+\frac{1}{2\pi i}\int_{(M-\varepsilon)}\frac{\Gamma(s_r+z)\Gamma(-z)}{\Gamma(s_r)}\zeta_{r-1}^{\rm des}(\Bold{$s$}'+\Bold{$z$}')\zeta_1^{\rm des}(-z)dz.
	\intertext{By the same arguments to those of \cite{Matsumoto}, the above second term converge. By using the fact that the residue of gamma function $\Gamma(s)$ at $s=-j$ is $\frac{(-1)^j}{j!}$, we get}
	=&\sum_{j=0}^{M-1}\binom{-s_r}{j}\zeta_{r-1}^{\rm des}(s_1,\dots,s_{r-2},s_{r-1}+s_r+j)\zeta_1^{\rm des}(-j) \\
	&+\frac{1}{2\pi i\Gamma(s_r)}\int_{(M-\varepsilon)}\Gamma(s_r+z)\Gamma(-z)\zeta_{r-1}^{\rm des}(\Bold{$s$}'+\Bold{$z$}')\zeta_1^{\rm des}(-z)dz. \\
\end{align*}
Setting $s_r=-k$ and $M=k+1$ for $k\in\mathbb{N}_0$, we obtain
\begin{equation*}
	\zeta_r^{\rm des}(s_1,\dots,s_{r-1},-k)=\sum_{j=0}^{k}\binom{k}{j}\zeta_{r-1}^{\rm des}(s_1,\dots,s_{r-2},s_{r-1}-k+j)\zeta_1^{\rm des}(-j),
\end{equation*}
because $1/\Gamma(-k)=0$ for $k\in\mathbb{N}_0$.
\end{proof}

\begin{rem}
	 In case of $r=2$, the above theorem recovers the equation
	\begin{equation*}
		\zeta_2^{\rm des}(s,-N)=\sum_{i+j=N}
		\binom{N}{i}\zeta_1^{\rm des}(s-i)\zeta_1^{\rm des}(-j)
	\end{equation*}
	shown in \cite[Proposition 4.3]{FKMT2}.
\end{rem}
By Proposition \ref{thm:3.1}, we obtain the following corollary.
\begin{prp}\label{crl:3.2}
	For $s_1,\dots,s_{r-1}\in\mathbb{C}$ and $l\in\mathbb{N}_0$, we have
	\begin{equation}\label{eqn:4.1}
		\zeta_{r-1}^{\rm des}(s_1,\dots,s_{r-1})\zeta_1^{\rm des}(-l)=\sum_{i+j=l}(-1)^i\binom{l}{i}\zeta_r^{\rm des}(s_1,\dots,s_{r-2},s_{r-1}-i,-j).
	\end{equation}
\end{prp}
\begin{proof}
	We prove this claim by induction on $l$. It is clear that the case of $l=0$ follows from the case of $k=0$ of Proposition \ref{thm:3.1}. By putting $k=l_0\ (\geq1)$ in the equation (\ref{eqn:4.2}), we get
	\begin{align*}
		\zeta_{r-1}^{\rm des}(s_1&,\dots,s_{r-1})\zeta_1^{\rm des}(-l_0) \\
		&=\zeta_r^{\rm des}(s_1,\dots,s_{r-1},-l_0)-\sum_{j=0}^{l_0-1}\binom{l_0}{j}\zeta_{r-1}^{\rm des}(s_1,\dots,s_{r-2},s_{r-1}-l_0+j)\zeta_1^{\rm des}(-j).
	\end{align*}
	In the second term of the right hand side of this equation, by using our induction hypothesis (i.e. the equation (\ref{eqn:4.1}) in the case of $0\leq l\leq l_0-1$), we obtain the equation (\ref{eqn:4.1}) of $l=l_0$.
\end{proof}
Putting $p=r-1$ and $q=1$ in Theorem \ref{prop:2.1}, we obtain
	\begin{equation*}
		\zeta_{r-1}^{\rm des}(-k_1,\dots,-k_{r-1})\zeta_1^{\rm des}(-l)= \sum_{i + j=l}(-1)^{i}\binom{l}{i} \zeta_{r}^{\rm des}(-k_1,\dots,-k_{r-2},-k_{r-1}-i,-j)
	\end{equation*}
for $k_1,\dots,k_{r-1},l\in\mathbb{N}_0$. Therefore the equation (\ref{eqn:4.1}) can be regarded as a generalization of this equation. 

\begin{rem}
	In our  forthcoming paper \cite{Komi2}, we will show a more general formula which extends  both Theorem \ref{prop:2.1} and Proposition \ref{crl:3.2}.
\end{rem}

\bigskip
\thanks{ {\it Acknowledgements}. The author is cordially grateful to Professor H. Furusho for guiding him towards this topic and for giving  useful suggestions to him. He greatly appreciates the referee's numerous and helpful comments. This work was supported by JSPS KAKENHI Grant Number JP18J14774.}


\end{document}